\documentclass{amsart}
\usepackage[leqno]{amsmath}
\usepackage{amsfonts,eucal,amssymb}
\newtheorem{Ex}{Example}
\newtheorem{Quest}[Ex]{Question}
\newtheorem{Prop}[Ex]{Proposition}
\newtheorem{Th}[Ex]{Theorem}
\newtheorem{Lemma}[Ex]{Lemma}
\newtheorem{Cor}[Ex]{Corollary}
\newtheorem{Def}[Ex]{Definition}
\renewcommand{\theequation}{\mbox{\thesection;\kern.17em\arabic{equation}}}
\author{
Angelo Bella}
\address{
Dipartimento di Matematica e Informatica\\
Universit\`a di Catania\\
Viale Andrea Doria 6\\
95125 Catania}

\email{bella@dmi.unict.it}

\author{
Camillo Costantini}

\address{
Dipartimento di Matematica\\
Universit\`a di Torino\\
via Carlo Alberto 10\\
10123 Torino -- Italy}

\email{camillo.costantini@unito.it}

\author{
Santi Spadaro}

\address{
Department of Mathematics and Statistics\\
Auburn University\\
221 Parker Hall\\
Auburn, Alabama -- USA 36849-5310}

\curraddr{ Department of Mathematics\\
Ben Gurion University of the Negev\\
Be'er Sheva\\
84105 Israel }

\email{santi@cs.bgu.ac.il}

\title{$P$-spaces and the Whyburn property}
\date{}
\def\A{\mathcal{A}}
\def\M{\mathcal{M}}
\def\I{\mathcal{I}}
\def\L{\mathcal{L}}
\def\S{\mathcal{S}}
\def\T{\mathcal{T}}
\newcommand{\imp}{\Rightarrow}

\newcommand{\dom}{\mathop{\mathrm{dom}}}
\newcommand{\N}{\mathbb N}

\newcommand{\by}{{\bar y}}
\newcommand{\bfin}[2]{{\big\{#1\,\colon\,#2\big\}}}
\newcommand{\Bfin}[2]{{\Big\{#1\,\colon\,#2\Big\}}}
\newcommand{\fin}[2]{{\{#1\,\colon\,#2\}}}

\newcommand{\fore}[2]{{\forall #1\in #2\colon}}
\newcommand{\ext}[2]{{\exists #1\in #2\colon}}

\newcommand{\om}{{\omega}}

\newcommand{\hal}{{\hat\al}}

\newcommand{\hx}{{\hat x}}

\newcommand{\cJ}{{\mathcal J}}    \newcommand{\cI}{{\mathcal I}}
\newcommand{\vuo}{{\emptyset}}

\newcommand{\bx}{{\bar x}}
\newcommand{\stm}{{\setminus}}
\newcommand{\al}{{\alpha}}     
\newcommand{\ep}{{\varepsilon}}
   
\newcommand{\si}{\subseteq}

\newcommand\cS{{\mathcal S}}   \newcommand\cF{{\mathcal F}}

\newcommand\cA{{\mathcal A}}   \newcommand\cL{{\mathcal L}}

\newcommand\fs{{\varphi}}

\newcommand{\slashedell}{}\let\slashedell\l
\newcommand{\ringaccent}{}\let\ringaccent\r
\renewcommand{\l}{\relax\ifmmode\mathopen\else\slashedell\fi}
\renewcommand{\r}{\relax\ifmmode\mathclose\else\expandafter\ringaccent\fi}
\newcommand{\m}{\relax\ifmmode\mathrel\else\undefined\fi}
\newcommand{\abs}[2][]{\csname#1l\endcsname|#2\csname#1r\endcsname|}
\newcommand{\norm}[2][]{\csname#1l\endcsname\|#2\csname#1r\endcsname\|}
\newcommand{\s}[2][]{\csname#1l\endcsname\{#2\csname#1r\endcsname\}}
\newcommand{\set}[3][]%
{\csname#1l\endcsname\{\,#2\csname#1m\endcsname|#3\,\csname#1r\endcsname\}}

\newcommand\new{\newcommand}
\new\bZ{{\bf Z}}   \new\bSg{{\bf\Sigma}}
\new\tet{{\tilde\eta}}
\new\hep{{\hat\ep}}   \new\dl{{\delta}}   \new\hdl{{\hat\delta}}
\new\hf{{\hat f}}   \new\hg{{\hat g}}  \new\hn{{\hat n}}
\new\hm{{\hat m}}
\new\hht{{\hat t}}
\new\hth{{\hat\tth}}
\new\xx{{x\in X}}   \new\mn{{m\in\N}}  \new\nn{{n\in\N}}
\new\het{{\hat\eta}}   \new\hC{{\hat C}}   \new\tas{{t^\ast}}
\new\tC{{\tilde C}}
\new\bat{\big(}   \new\bct{\big)}
\new\tth{{\vartheta}}
\new\sn{{(Y,\Vert\cdot\Vert)}}   %\new\norm{{\Vert\cdot\Vert}}
\new\zz{{\bf 0}} \new\bby{{{\bf y}}}   \new\tn{{\tilde n}}
\new\tr{{\tilde r}}
\new\bbv{{\bf v}}   \new\bbu{{\bf u}}
\new\aep{{\ep^\ast}}   \new\sE{{E^\star}}   \new\aE{{E^\ast}}
\new\tae{{\tilde E^\ast}}
\new\ud[1]{{{\mathcal{UD}}_{#1}}}   \new\udm[1]{{{\mathcal{UD}}_{#1}^{\rm max}}}
\new\car[1]{{\vert#1\vert}}    \new\bcar[1]{{\big\vert#1\big\vert}}
\new\shE{{E^\sharp}}   \new\bbby{{\bar\bby}}
\new\sql{{\sqsubseteq}}      \new\sqs{{\sqsubset}}
\new\hxi{{\hat\xi}}     
\new\hY{{\hat Y}}   \new\hX{{\hat X}}  \new\hro{{\hat\rho}} 
\new\noi{\noindent}   \new\tA{{\tilde A}}  

\new\now[1]{{\Vert#1\Vert_{\rm w}}}
\new\bnow[1]{{\big\Vert#1\big\Vert_{\rm w}}}
\new\roi{{\rho_\infty}}    \new\row{{\rho_{\rm w}}}
\new\leff{{\preccurlyeq_f}}    \new\sleff{{\prec_f}}
\new\lef{\mathrel{\preccurlyeq_f}}   \new\slef{\mathrel{\prec_f}}
\new\sle{\sqsubseteq}
\new\zan{{\{0,1,\dots,n\}}}   \new\uan{{\{1,\dots,n\}}}
\new\tD{{\tilde D}}
\new\bBigl{\text{\raisebox{-.27ex}{\begin{Large}(\end{Large}}}\mspace{-2mu}}
\new\bBigr{\mspace{-2mu}\text{\raisebox{-.27ex}{\begin{Large})\end{Large}}}}
\new\mom{{{}^{<\om}\om_1}}
\new\omm{{{}^\om\om_1}}
\new\bF{{\mathbf{F}}}
\new\uph{{\upharpoonright}}

\begin{document}
\maketitle
\bibliographystyle{abbrv}
\begin{abstract}
\noi
We investigate the Whyburn and weakly Whyburn property in the class of $P$-spaces, that is spaces where every countable intersection of open sets is open. We construct examples of non-weakly Whyburn $P$-spaces of size continuum, thus giving a negative answer under CH to a question of Pelant, Tkachenko, Tkachuk and Wilson (\cite{PTTW}). In addition, we show that the weak Kurepa Hypothesis (an assumption weaker than CH) implies the existence of a non-weakly Whyburn $P$-space of size $\aleph_2$. Finally, we consider the behavior of the above-mentioned properties under products; we show in particular that the product of a Lindel\"of weakly Whyburn
P-space and a Lindel\"of Whyburn $P$-space is weakly Whyburn, and we give a consistent example of a non-Whyburn product of two Lindel\"of Whyburn $P$-spaces.
\end{abstract}

\section{Introduction}
We assume all spaces to be Hausdorff. All undefined notions can
be found in \cite{E} and \cite{K1}.

Whyburn and weakly Whyburn spaces have recently been considered
by various authors. They provide a natural generalization of
Fr\'echet and sequential spaces which is weak enough to offer
lots of challenges, and indeed some fundamental questions are
still open. They have also turned out to be useful in the
$M_3=M_1$ problem (see \cite{GT} and \cite{TY}). 

Let $X$ be a topological space. A set $F \subseteq X$ is said to be an \emph{almost closed} converging to $x$ if $\overline{F} \setminus F=\{x\}$; if this happens we will write $F \to x$. A space $X$ is said to be \emph{Whyburn}
(\emph{weakly Whyburn}) if for each non closed set $A \subseteq X$
and each (some) point $x \in \overline{A} \setminus A$ there is
an almost closed set $F \subseteq A$ such that $F \to x$. 

Fr\'echet spaces are Whyburn and sequential spaces are weakly
Whyburn. The class of weakly Whyburn spaces is wide enough to
include all regular scattered spaces.

Recall that a space is radial (pseudoradial) if for each $A
\subseteq X$ and every (some) point $x \in \overline{A} \setminus
A$ there is a transfinite sequence $\{x_\alpha: \alpha \in
\kappa\} \subseteq A$ which converges to $x$.  The radial character
of a pseudoradial space $X$ is the smallest cardinal $\kappa $
such that  the  definition of pseudoradiality for $X$  works  by
considering sequences of length at most $\kappa$.

The (weakly) Whyburn property behaves very nicely on compact
spaces; namely, every countably compact Whyburn space is
Fr\'echet and every compact weakly Whyburn space is pseudoradial.

A space is called a \emph{$P$-space} if every $G_\delta$ set is
open. A Lindel\"of space is one in which every open cover has a
countable subcover. Lindel\"of $P$-spaces exhibit a behaviour that
is somewhat very close to that of compact spaces; for example,
every Lindel\"of subspace of a Lindel\"of $P$-space is closed and every Lindel\"of Hausdorff $P$-space is normal.
Notice, however, that the class of Lindel\"of $P$-spaces is only
finitely productive.

The study of the Whyburn and weakly Whyburn property on $P$-spaces was initiated in \cite{PTTW}, where it is proved that every $P$-space of character $\omega_1$ is Whyburn and that not every regular Lindel\"of $P$-space is Whyburn, and a bunch of interesting problems is left open. 

Whyburn and weakly Whyburn spaces were formerly known as AP (WAP) spaces (see \cite{S1} and \cite{PT}). Although this terminology is still in use today (see for example \cite{GT}), we prefer to adopt the new name, because it sounds better and gives credit to a paper of Whyburn where these spaces were first studied (\cite{W}).

\section{Toward a small non weakly Whyburn $P$-space}
The main purpose of this section is to give several examples of non weakly Whyburn 
$P$-spaces. No example of this kind has been exhibited so far. We are particularly concerned with finding examples of minimal cardinality or with further restrictions (for example, Lindel\"of examples). Problem 3.6 in \cite{PTTW} asks whether every regular $P$-space of cardinality $\aleph_1$ is (weakly) Whyburn. Let's start by settling the weak form of this problem with a very simple ZFC construction.

\begin{Ex}
There is a zero-dimensional, non-Whyburn regular (hence, completely regular) 
$P$-space of cardinality $\omega_1$.
\end{Ex}
\begin{proof}
For each $\alpha \in \omega_1$ let $L_\alpha$ be a copy of the one-point 
lindel\"ofication of a discrete space of cardinality $\omega_1$,
$p_\alpha$ being the unique non isolated point of $L_\alpha$ (whose neighbourhoods
are all co-countable subsets of $L_\al$ containing the point $p_\al$ itself). Let $Y=\bigoplus_{\alpha\in \omega_1} L_\alpha$ and choose a point $\infty \notin Y$.
Topologize $X=Y \cup \{\infty\}$ in the following way. Neighbourhoods of
the points of $Y$ are as usual, while a neighbourhood of $\infty$ is
$\{\infty\} \cup \bigcup_{\alpha \geq \beta} T_\alpha$ for some $\beta \in \omega_1$ and some choice of a neighbourhood $T_\al$ of $p_\al$ in $L_\al$ for every $\al\ge\beta$. It is clear that $X$ is a zero-dimensional regular $P$-space of cardinality $\omega_1$. To show that it is not Whyburn consider the set $A=Y\setminus \{p_\alpha : \alpha \in \omega_1\}$ and note that $\infty \in \overline{A}$ and if $B \subseteq A$ is a set having the point $\infty$ in its closure then $\overline{B}$
contains infinitely many points $p_\alpha$.
\end{proof}
Notice that the existence of a non-Whyburn $P$-space of cardinality $\aleph_1$ and character less than $2^{\aleph_1}$ is not provable in ZFC, see Theorem 3 in \cite{BS} and the remark at the end of the section. We now come to the main dish of this section.

\begin{Ex} \label{Ex: nonwwh} 
There is a completely regular non-weakly Whyburn $P$-space of cardinality 
$\mathfrak{c}$.
\end{Ex}
\begin{proof}  
Let $Y=\mom$, $Z=\omm$ and $X=Y\cup Z$. We will introduce a topology on $X$ in the following way. Set $\Gamma={}^Y\om_1={}^{(\mom)}\om_1$, and for every $g\in\Gamma$ and $\fs\in Y$ put
\[
V_{\fs,g}=\{\psi\in X\,\vert\,\psi\text{ is an extension }\text{of }\fs
\wedge\ \fore n{(\dom\psi\stm\dom\fs)}\psi(n)\ge g(\psi\uph_n)\}\eqno(\ast)
\] 
(observe, in particular, that $\fs$ always belongs to $V_{\fs,g}$).\par 
First of all, we want to show that associating to every $\fs\in Y$ the collection $\{V_{\fs,g}
\,\colon\,g\in\Gamma\}$ as a fundamental system  of (open) neighbourhoods, and letting every $\fs\in Z$ to be isolated, gives rise to a topology of $P$-space on $X$. Actually, we have to show that
\begin{enumerate}
\item $\fore\fs Y\fore g\Gamma\fore\psi{V_{\fs,g}\cap Y}\ext{g'}\Gamma V_{\psi,g'}\si V_{\fs,g}$;
\item $\fore\fs Y\fore{\Gamma'}{[\Gamma]^{\le\om}}\ext g\Gamma V_{\fs,g}\si\bigcap_{g'\in\Gamma'}V_{\fs,g'}$.
\end{enumerate}\par
To prove (1), simply take $g'=g$; we claim that $V_{\psi,g}\si V_{\fs,g}$. Indeed, if $\eta\in V_{\psi,g}$, then first of all $\eta$ extends $\psi$; since $\psi$, in turn, extends $\fs$ (as $\psi\in V_{\fs,g}$), we have that $\eta$ extends $\fs$. Therefore, to prove that $\eta\in V_{\fs,g}$, we only have to show that $\eta(n)\ge g(\eta\uph _n)$ for every $n\in\dom\eta\stm\dom\fs$. Now, if $n\in\dom\eta\stm\dom\fs$, then either $n\in\dom\psi\stm\dom\fs$, in which case $\eta(n)=\psi(n)\ge g(\psi\uph_n)=g(\eta\uph_n)$ (as $\eta$ extends 
$\psi$ and $\psi\in V_{\fs,g}$), or $n\in\dom\eta\stm\dom\psi$, in which case we have 
$\eta(n)\ge g(\eta\uph_n)$ simply because $\eta\in V_{\psi,g}$.\par
As for property (2), let $\hat\fs\in Y$ and $\Gamma'$ be a countable subset of 
$\Gamma$, and define $g\in\Gamma$ by letting $g(\fs)=\sup\fin{g'(\fs)}{g'\in\Gamma'}$.
Now, for each $\fs\in Y$, if $\psi$ is an arbitrary element of $V_{\hat\fs,g}$ and $\hat g'$ is an arbitrary element of $\Gamma'$, then to prove that $\psi\in V_{\hat\fs,\hat g'}$ we just have to show that $\psi(n)\ge\hat g'(\psi\uph_n)$ for every $n\in \dom\psi\stm\dom\hat\fs$. Indeed, if $n\in \dom\psi\stm\dom\hat\fs$, then from $\psi\in V_{\hat\fs,g}$ it follows that $\psi(n)\ge g(\psi\uph_n)=\sup\fin{g'(\psi\uph_n)}{g'\in\Gamma'}\ge\hat g'(\psi\uph_n)$.\par
Let $\tau$ be the topology generated on $X$ by associating to every $\fs\in X$ the fundamental system of (open) neighbourhoods $\cI_\fs$, defined as
\[
\cI_\fs=
\begin{cases}
\fin{V_{\fs,g}}{g\in\Gamma}&\text{if $\fs\in Y$;}\\
\big\{\{\fs\}\big\}&\text{if $\fs\in Z$.}
\end{cases}
\]
We want to prove now that $(X,\tau)$ is a completely regular space. To this end, we will first show that $(X,\tau)$ is T$_0$, and then that for every $\fs\in X$ and every $V\in
\cI_\fs$, $V$ is a closed (and open) set. As for the T$_0$ property, let $\fs',\fs''$ be distinct elements of $X$; then one of them, say $\fs'$, has domain not less than that of the other element. This implies in particular (as $\fs'\neq\fs''$) that $\fs''$ cannot be an extension of $\fs'$; therefore, fixing any $g'\in\Gamma$, it follows from $(\ast)$ that 
$V_{\fs',g'}$ is a neighbourhood of $\fs'$ which does not contain $\fs''$. Let now show that every collection $\cI_\fs$ consists of closed sets. If $\fs\in Z$, then $\{\fs\}$ is the only element of $\cI_\fs$; thus, consider an arbitrary $\psi\in X\stm\{\fs\}$. Of course, if $\psi\in Z$ then $\{\psi\}$ is a neighbourhood of $\psi$ disjoint from $\{\fs\}$; therefore, we may assume $\psi\in Y$, and consider an $\hn\in\om$ with $\hn\ge\dom\psi$. Also, fix an element $\hg$ of $\Gamma$ such that $\hg(\fs\uph_\hn)>\fs(\hn)$; then it follows from 
$(\ast)$ that $\fs\notin V_{\psi,\hg}$, i.e. $\{\fs\}\cap V_{\psi,\hg}=\vuo$. Suppose now that $\fs\in Y$. Let $g$ be any element of $\Gamma$, and consider an arbitrary $\psi\in X\stm V_{\fs,g}$---also, since the case $\psi\in Z$ is again trivial, we may assume that $\psi\in Y\stm V_{\fs,g}$. By $(\ast)$, the fact that $\psi\notin V_{\fs,g}$ means that $\psi$ is not an extension of $\fs$, or that there exists $\hn\in\dom\psi\stm\dom\fs$ with $\psi(\hn)<g(\psi\uph_\hn)$. If we are in this second case, then taking any $g^\ast\in\Gamma$
we see that every $\psi'\in V_{\psi,g^\ast}$ is in particular an extension of $\psi$, so that $\hn\in\dom\psi\stm\dom\fs\si\dom\psi'\stm\dom\fs$ and $\psi'(\hn)=\psi(\hn)<g(\psi\uph_\hn)=g(\psi'\uph_\hn)$; as a consequence, $V_{\psi,g^\ast}\cap V_{\fs,g}=\vuo$. Suppose now that $\psi$ is not an extension of $\fs$. If $\fs$ is not an extension of 
$\psi$, either, then there exists $\hn\in\dom\fs\cap\dom\psi$ with $\fs(\hn)\neq\psi(\hn)$; of course, this entails that $\fs$ and $\psi$ cannot have any common extension, hence in particular taking any $g^\ast\in\Gamma$ it will follow again that $V_{\psi,g^\ast}\cap V_{\fs,g}=\vuo$. If, on the contrary, $\fs$ is a (proper) extension of $\psi$, then let 
$\hn=\dom\psi$ and fix a $g^\ast\in\Gamma$ such that $g^\ast(\psi)>\fs(\hn)$; since every $\eta\in V_{\fs,g}$ is in particular an extension of $\fs$, we will have the inequality 
$\eta(\hn)=\fs(\hn)<g^\ast(\psi)=g^\ast(\fs\uph_\hn)=g^\ast(\eta\uph_\hn)$---so that, by $(\ast)$, $\eta\notin V_{\psi,g^\ast}$. Therefore, we still see that $V_{\psi,g^\ast}\cap V_{\fs,g}=\vuo$.\par
Finally, we show that $(X,\tau)$ is not weakly Whyburn. First of all, notice that $Z$ is not closed in $(X,\tau)$; actually, $Z$ turns out to be even dense. Indeed, let $\fs$ be any element of $Y$, and $V_{\fs,g}$ (with $g\in\Gamma$) be an arbitrary basic neighbourhood of $\fs$. Then put $\hn=\dom\fs$, and define by induction an $\al_n\in\om_1$ for every $n\ge\hn$ in the following way.\par\noi
--- $\al_\hn=g(\fs)$;\par\noi
--- $\al_{n+1}=g(\fs\,{}^{\smallfrown}\langle\al_\hn,\al_{\hn+1},\dots,\al_n\rangle)$.\par\noi
Now, let $\psi\in Z$ be defined by
\[
\psi(n)=
\begin{cases}
\fs(n)&\text{for $n<\hn$;}\\
\al_n&\text{for $n\ge\hn$;}
\end{cases}
\]
clearly, $\psi$ extends $\fs$, and for every $n\ge\hn$ (i.e., for every $n\in\dom\psi\stm\dom\fs$) we have that $\psi(n)=\al_n=g(\psi\uph_n)$. Therefore, it follows from $(\ast)$ that $\psi\in V_{\fs,g}$, so that $V_{\fs,g}\cap Z\neq\vuo$.\par
Thus, we will conclude the proof that $(X,\tau)$ is not weakly Whyburn by showing that if $M\si Z$ and $\hat\fs\in Y\cap\overline M$, then there exists an element of $Y\stm\{\hat\fs\}$ which belongs to $\overline M$. Actually, we will show more precisely that there exists an $\al\in\om_1$ such that $\hat\fs\,{}^{\smallfrown}\langle\al\rangle\in\overline M$. As a matter of fact, toward a contradiction, suppose that for every $\al\in\om_1$ there exists $g_\al\in\Gamma$ such that 
\[
V_{\hat\fs\,{}^{\smallfrown}\langle\al\rangle,g_\al}\cap M=\vuo.\eqno(\sharp)
\]
Then put $\hn=\dom\hat\fs$ and consider a $\hg\in\Gamma$ such that 
\[
\fore\fs Y\big(\dom\fs\ge\hn+1\Longrightarrow\hg(\fs)=g_{\fs(\hn)}(\fs)\big).\eqno(\spadesuit)
\]
Since $\hat\fs\in\overline M$, there must exist an $\eta\in V_{\hat\fs,\hg}\cap M$, and this $\eta$ will be in particular an extension of 
$\hat\fs$---hence also of $\hat\fs\,{}^{\smallfrown}\langle\eta(\hn)\rangle$. However, since $V_{\hat\fs\,{}^{\smallfrown}\langle\eta(\hn)\rangle, g_{\eta(\hn)}}\cap M=\vuo$
(by $(\sharp)$), we have in particular that $\eta\notin V_{\hat\fs\,{}^{\smallfrown}\langle\eta(\hn)\rangle, g_{\eta(\hn)}}$. But $\eta$ is an extension of $\hat\fs\,{}^{\smallfrown}\langle\eta(\hn)\rangle$, thus by $(\ast)$ there must exist an $n^\ast\in\om\stm\dom\big(
\hat\fs\,{}^{\smallfrown}\langle\eta(\hn)\rangle\big)=\om\stm(\hn+1)$ such that $\eta(n^\ast)<g_{\eta(\hn)}\big(\eta\uph_{n^\ast}\big)$. Thus, letting $\fs^\ast=\eta\uph_{n^\ast}$, it follows from $(\spadesuit)$ (as $\dom\fs^\ast=n^\ast\ge\hn+1$) that $\eta(n^\ast)<
g_{\eta(\hn)}\big(\eta\uph_{n^\ast}\big)=g_{\fs^\ast(\hn)}(\fs^\ast)=\hg(\fs^\ast)=\hg(\eta\uph_{n^\ast})$. Clearly, this contradicts the fact that $\eta\in V_{\hat\fs,\hg}$ (take again $(\ast)$ into account).  
\end{proof}

Here is another way to construct a non weakly Whyburn $P$-space of cardinality $\mathfrak{c}$. Murtinov\'a (\cite{M}) has come up independently with a similar example. We will exploit the following example of a Baire $P$-space and a trick that Gruenhage and Tamano (\cite{GT}) used to get a countable stratifiable non-weakly Whyburn space.

\begin{Ex} \label{meager}
There exists a $T_1$ regular dense-in-itself $P$-space of cardinality continuum in which every meager set is nowhere dense.
\end{Ex}
\begin{proof}
Recall that the \emph{support} of a function is the greatest subset of its domain where the function never vanishes. Let $X=\{f\colon\omega_1\to\omega\ \colon\,f\mathrm{\ has
\  \emph{countable}\ support}\}$. Set $B(f, \alpha)=\{g \in X\ \colon\ \forall\beta\leq\alpha\colon g(\beta)=f(\beta)\}$. We define a topology on $X$ by declaring $\{B(f,\alpha): \alpha \in \omega_1\}$ to be a fundamental system of neighbourhoods at $f$. It is easy to see that $X$ is a dense-in-itself $P$-space and that $X$ is $T_1$ and zero-dimensional. Let $\{N_n: n \in \omega \}$ be a countable family of nowhere dense subsets of $X$. Let $f \in X$ and $\alpha \in \omega_1$. Since $N_0$ is nowhere dense, there must be a function $g_0 \in B(f, \alpha)$ and an ordinal $\alpha_0$ such that $B(g_0, \alpha_0) \cap N_0=\emptyset$. Since $B(f, \alpha)$ is an open set we can suppose $B(g_0, \alpha_0) \subset B(f, \alpha)$, and since $\{B(g_0, \gamma) : \gamma \in \omega_1 \}$ is a decreasing family, we can suppose that $\alpha_0>\alpha$. Suppose we have found points $\{g_i : i<k \}$ and increasing ordinals $\{\alpha_i : i <k \}$ in such a way that $B(g_j, \alpha_j) \subset B(g_{j-1}, \alpha_{j-1})$ and $B(g_j, \alpha_j) \cap N_j=\emptyset$ for each $j<k$. Since $N_k$ is nowhere dense there must be a point in $g_k \in B(g_{k-1}, \alpha_{k-1})$ such that $g_k \notin \overline{N_k}$, and thus there must be $\alpha_k$ such that $B(g_k, \alpha_k) \cap N_k=\emptyset$. Again we can suppose $B(g_{k}, \alpha_k) \subset B(g_{k-1}, \alpha_{k-1})$ and $\alpha_k>\alpha_{k-1}$. At the end of the induction the set $\bigcap \{B(g_i, \alpha_i) : i \in \omega \}$ will be a nonempty open set contained in $B(f, \alpha)$, and disjoint from $\bigcup \{N_n : n \in \omega \}$ (to check that $\bigcap\fin{B(g_i,\al_i)}{i\in\om}\neq\vuo$, consider that such a set includes $B(\hat g,\hal)$, where $\hal=\sup\fin{\al_i}{i\in\om}$ and $\hat g$ is any element of $X$ such that $\hat g(\al')=g_i(\al')$ for every $\al'\in\om_1$ and $i\in\om$ with $\al'\le\al_i$). Picking any point $h \in \bigcap \{B(g_i, \alpha_i) \,\colon\, i \in \omega \}$ we see that $h$ is not in the closure $\bigcup_{n \in \omega} N_n$, therefore $B(f, \alpha)$ cannot be contained in $\overline{\bigcup_{n \in \omega} N_n}$. Since the choice of $f$ and $\alpha$ was arbitrary it turns out that $\bigcup_{n \in \omega} N_n$ is nowhere dense.
\end{proof}

Recall that a \emph{weak $P$-space} is a space in which every countable set is closed.

\begin{Ex} \label{Ex: nonwwh2}
There are a non-weakly Whyburn $P$-space of cardinality $\mathfrak{c}$ and a non weakly-Whyburn weak $P$-space of cardinality $\omega_1$ in ZFC.
\end{Ex}

\begin{proof}
Let $X_1$ and $X_2$ be copies of the space $X$ in Example $\ref{meager}$, $Y_1$ and $Y_2$ be copies of the space $Y$ described in the above remark and $f$ be a bijection between either $X_1$ and $X_2$ or $Y_1$ and $Y_2$ (the context will clarify which spaces we mean). Put $Z=X_1 \cup X_2$ and $W=Y_1 \cup Y_2$. Take each point of $X_2$ ($Y_2$) to be isolated and declare a neighbourhood of $x$ in $X_1$ ($Y_1$) to be of the form $U \cup f(U)\setminus f(N)$, where $U$ is a neighbourhood of $x$ in $X$ ($Y$), and $N$ is a nowhere dense set in $X_1$ ($Y_1$).

\emph{$Z$ is a $P$-space:} Let $\mathcal{U}=\{U_n \cup (f(U_n) \setminus f(N_n)) : n \in \omega\}$ be a countable family of neighbourhoods of a point of $z \in X_1$. Then $\bigcap \mathcal{U}= \bigcap_{n \in \omega} U_n \cap \bigcap_{i \in \omega} f(U_i) \setminus \bigcup_{k \in \omega} f(N_k)$. The fact that $X$ is a $P$-space where every meager set is nowhere dense implies that the last set is a neighbourhood of $z$ in $Z$. Thus every point of $Z$ is a $P$-point, i.e., $Z$ is a $P$-space.

\emph{$W$ is a weak $P$-space:} Indeed, let $C \subset W$ be a countable set and $x \in Y_1\stm C$. Let $C_1=C \cap Y_1$ and $C_2=C \cap Y_2$. If $x \notin f^{-1}(C_2)$ then  let $U$ be a neighbouthood of $x$ in $Y_1$ disjoint from $C_1 \cup f^{-1}(C_2)$. Then $U \cup f(U)$ will be a neighbourhood of $x$ in $W$ disjoint from $C$. If $x \in f^{-1}(C_2)$ let $U$ be a neighbourhood of $x$ disjoint from $C_1$, and as $Y$ is a $P$-space, the countable set $f^{-1}(C_2)$ will be nowhere dense in it (in fact discrete) so $U \cup (f(U) \setminus C_2)$ is a neighbourhood of $x$ in $W$ disjoint from $C$.

\emph{$Z$ and $W$ are not weakly Whyburn:} The set $X_2$ is not closed, however there is no almost closed set converging outside $X_2$. In fact if $A \subset X_2$ is a 
non-closed set in $Z$ then $f^{-1}(A)$ cannot be nowhere dense, hence $Int(\overline{f^{-1}(A)}) \neq \emptyset$. Moreover this last set must be uncountable and dense-in-itself, or otherwise $X$ would have isolated points. By the definition of the topology on $Z$, every point of $Int(\overline{f^{-1}(A)})$ will be in the closure of $A$. This proof also works for the space $W=Y_1 \cup Y_2$.
\end{proof}

Thus, under [CH] we get a consistent negative answer to Problem 3.6 in \cite{PTTW}.

\begin{Quest} \label{quest}
Does there exist a model of ZFC + not CH in which there is a Baire $P$-space of cardinality $\omega_1$?
\end{Quest}

It is easy to realize that the fact that every countable union of nowhere dense sets is nowhere dense is equivalent to being a Baire space in the realm of $P$-spaces.

There is no hope of getting a ZFC example like in Question $\ref{quest}$ because the non-existence of a Baire space of cardinality $\omega_1$ is known to be consistent (see \cite{ST}).

By means of the contruction of Example $\ref{Ex: nonwwh2}$, a positive answer to the previous question would guarantee the existence of a non-weakly Whyburn $P$-space of cardinality $\omega_1$ in a model where the continuum hypotheses fails. Yet, the following problem would still be open.

\begin{Quest}
Is there a ZFC example of a non-weakly Whyburn $P$-space of cardinality $\aleph_1$?
\end{Quest}

In \cite{PTTW} it was remarked that every Lindelof $P$-space of cardinality $\aleph_1$ is Whyburn. One may wonder if the same applies to weak $P$-spaces.

\begin{Th}[{[CH]}]
There exists a regular Lindel\"of non weakly Whyburn weak $P$-space of cardinality and character $\aleph_1$.
\end{Th}

\begin{proof}
Let $Y$ be the real line with the density topology (see for example \cite{T}). This space is known to be a Tychonoff weak $P$-space. Also, since it is a dense-in-itself ccc Baire space having $\pi$-weight $\aleph_1$, by a well-known argument (see for example \cite{R}) it has a Luzin subspace $Z \subset Y$, i.e. a subspace in which every nowhere dense set is countable. Let $X$ be the Alexandroff duplicate of $Z$. It is clear that $X$ is a weak $P$-space, it is Lindel\"of because $Z$ is, and an argument similar to the one in 
Example~\ref{Ex: nonwwh2} shows that $X$ is not weakly Whyburn.
\end{proof}

\begin{Quest}
Can CH be removed from the previous theorem?
\end{Quest}

Going back to non weakly Whyburn $P$-spaces, we are now going to produce an example of cardinality $\aleph_2$ by using a set theoretic axiom weaker than CH. 
A tree $T$ of  cardinality $\omega_1$ with $\omega_2$-many
uncountable branches is called a {\it weak Kurepa tree}. Using a weak
Kurepa tree $T$ is non difficult to construct a regular $P$-space
$X$ of weight $\omega_1$ and cardinality $\omega_2$; take as
points of $X$ the uncountable branches and as a base the sets of the
form $\hat t=\{B\in X :t\in B\}$ for each $t\in T$. 

Unfortunately, the existence of a weak Kurepa tree is not
provable in ZFC (see Section 8  in \cite {We}).  The {\it weak Kurepa
Hypothesis}, briefly wKH, is the assertion that there exists a
weak Kurepa tree. In \cite{We} it is pointed out that CH implies
wKH, but  wKH is somehow much weaker than CH.

\begin{Ex}[{[{\sf wKH}]}]\label{Ex: nonwwh2'} 
There is a non-weakly Whyburn regular $P$-space of cardinality $\omega_2$.
\end{Ex}
\begin{proof} Let $X$ be a regular $P$-space of weight $\omega_1$ and cardinality
$\omega_2$  and let  $Y=\omega_2\cup\{p\}$,  topologized in such
a way that  every subset of
$\omega_2$ is open and  a set  $U\ni p$ is open whenever
$|Y\setminus U|\le \omega_1$. Fix a one-to-one  mapping
$f:Y\rightarrow X$ and let $Z=Y\times X$.  Obviously $Z$ is a
regular $P$-space of cardinality $\omega_2$. We claim that $Z$ 
does
not have the weak Whyburn property. Let $A=\{(\alpha ,f(\alpha
)) : \alpha \in \omega_2\}\subseteq Z$. Since the space
$X$ has weight $\omega_1$, the set $T$ of all  complete
accumulation points of  the set $\{f(\alpha ) :\alpha \in
\omega_2\}$ has cardinality $\omega_2$. It is easy to
check that  $\{p\}\times T\subseteq \overline A$  and so $A$ is
not closed in $Z$. Let $B=\{(\alpha ,f(\alpha )) :\alpha \in C\}$
be a subset of $A$.  A point can be in $\overline B\setminus A$
only if it is of the form $(p,x)$ and this may happen  only if
$|B|=\omega_2$. As before,  the set $T'$ of all complete
accumulation points of the set $f(C)$ has cardinality $\omega_2$
and  $\{p\}\times T'\subseteq \overline B\setminus A$.  
Therefore, no almost closed subset of $A$ can converge to a point
outside $A$ and we conclude that $Z$ does not have the weak
Whyburn property. 
\end{proof}

\begin{Quest}
Is there a ZFC example of a non-weakly Whyburn $P$-space of cardinality $\omega_2$?
\end{Quest}

The next result, which is of independent interest, shows that a positive answer to the previous question would require a different approach than that of our example $\ref{Ex: nonwwh2}$. We are going to use the following simple lemma.

\begin{Lemma}\label{1wkt}
Every regular $P$-space $X$ with $w(X)\le\omega_1$ has a base of cardinality not greater than 
$\omega_1$, consisting of clopen sets.
\end{Lemma}
\begin{proof}
By \cite[Theorem~1.1.15]{E}, it suffices to show that the clopen sets form a base for $X$, and this holds because every regular $P$-space is $0$-dimensional. 
\end{proof}

\begin{Prop}
If there exists a regular $P$-space $X$ with $\vert X\vert\ge\om_2$ and $w(X)\le\om_1$, then there exists a weak Kurepa tree.  
\end{Prop}
\begin{proof}
By the above lemma, there exists a base $\cL=\fin{L_\al}{\al\in\om_1}$ for $X$, consisting of clopen sets. Using transfinite induction, we will construct an 
$\om_1$-sequence $\big\{\cA_\al\big\}_{\al\in\om_1}$ of collections of nonempty subsets of $X$, with the following properties:
\begin{enumerate}
\item $\fore\al{\om_1}\cA_\al\text{\ is\ a\ (cl)open\ partition\ of\ }X$ not containing the empty set;
\item $\fore\al{\om_1}\forall\al'<\al\colon\cA_{\al}\text{\ is\  a\ refinement\ of\ }\cA_{\al'}$.
\end{enumerate}
\par
Let $\cA_0=\{X\}$. If $\cA_\al$ has been defined for a given $\al\in\om_1$, then let
\[
\cA_{\al+1}=(\fin{A\cap L_\al}{A\in\cA_\al}\cup\fin{A\stm L_\al}{A\in\cA_\al})\stm\{\vuo\}.
\]
Finally, if $\lambda\in\om_1$ is limit and $\cA_\al$ has been defined for $\al<\lambda$, then put
\[
\cA_\lambda=\Bfin{\bigcap_{\al\in\lambda}A_\al}{\fore\al\lambda A_\al\in\cA_\al\ \ \wedge\ \ \fore{\al',\al''}\lambda(\al'<\al''\imp A_{\al'}\supseteq A_{\al''})}\stm\{\vuo\}.
\]
We prove both property (1) and (2) by transfinite induction on $\al\in\om_1$.
As for (1), it is trivial when $\al=0$, and if it holds for a given $\hal\in\om_1$ then it is easily seen to hold also for $\hal+1$. Consider now a nonzero limit $\lambda\in\om_1$.  
Given an arbitrary $x\in X$, take for every $\al<\lambda$ an $A_\al\in\cA_\al$ such that 
$x\in A_\al$; since (by the inductive hypothesis) (2) holds for every $\al<\lambda$, we have that for every $\al'<\al<\lambda$ there exists a $\hat A\in\cA_{\al'}$ such that 
$A_\al\si\hat A$. This implies (as $\cA_{\al'}$, still by the inductive hypothesis, is a partition of $X$) that $A_\al\cap A=\vuo$ for every $A\in\cA_{\al'}\stm\big\{\hat A\big\}$; thus, since $A_{\al'}\in\cA_{\al'}$ and $A_\al\cap A_{\al'}\neq\vuo$ (as $x$ belongs to both sets), we must have that $A_{\al'}=\hat A$, hence $A_\al\si A_{\al'}$. Therefore, since of course $\bigcap_{\al\in\lambda}A_\al\neq\vuo$ (as such a set contains $x$), we have by the definition of $\cA_\lambda$ that $\bigcap_{\al\in\lambda}A_\al\in\cA_\lambda$, so that we have found an element of $\cA_\lambda$ which contains $x$. Therefore, 
$\cA_\lambda$ covers $X$. Since $\vuo\notin\cA_\lambda$ by definition, and every element of $\cA_\lambda$ is open as a countable intersection of open sets in a $P$-space,  
it only remains to show that $A'\cap A''=\vuo$ for distinct $A',A''\in\cA_\lambda$. Thus, suppose $A'=\bigcap_{\al\in\lambda}A'_\al$ and $A''=\bigcap_{\al\in\lambda}A''_\al$, with $A'_\al,A''_\al\in\cA_\al$ for $\al\in\lambda$, and $\al\mapsto A'_\al$ and $\al\mapsto 
A''_\al$ decreasing; of course, the fact that $A'\neq A''$ implies that for at least one $\hal\in\lambda$ we have that $A'_\hal\neq A''_\hal$. Since $\cA_\hal$, by the inductive hypothesis, is a partition, this implies in turn that $A'_\hal\cap A''_\hal=\vuo$; hence $A'\cap A''=\big(\bigcap_{\al\in\lambda}A'_\al\big)\cap\big(\bigcap_{\al\in\lambda}A''_\al
\big)=\vuo$, too.\par
Now we prove (2). Of course, for $\al=0$ there is nothing to prove. If (2) holds for every 
$\al$ less than a successor ordinal $\hal+1\in\om_1$, and $\al^\ast<\hal+1$, then either $\al^\ast=\hal$, in which case it is straightforward to realize from the general definition of 
$\cA_{\al+1}$ that $\cA_{\hal+1}$ is a refinement of $\cA_\hal=\cA_{\al^\ast}$, or $\al^\ast<\hal$, in which case since $\cA_{\hal+1}$ is a refinement of $\cA_\hal$ and 
$\cA_\hal$ (by the inductive hypothesis) is a refinement of $\cA_{\al^\ast}$ we also have that $\cA_{\hal+1}$ is a refinement of $\cA_{\al^\ast}$. Finally, if $\lambda\in\om_1$
is a limit ordinal, and $A$ is an arbitrary element of $\cA_\lambda$, then $A=\bigcap_{\al\in\lambda}A_\al$ with $A_\al\in\cA_\al$ for every $\al<\lambda$; therefore, for every $\al\in\lambda$, $A_\al$ is an element of $\cA_\al$ which include $A$.\par
Let us also point out a further property of the partitions $\cA_\al$ which immediately follows from the general definition of $\cA_{\al+1}$, and which will play a momentous 
r\^ole below:
\[
\fore\al{\om_1}\fore A{\cA_{\al+1}}(A\cap L_\al=\vuo\ \vee\ A\si L_\al).\eqno(\ast)
\]\par
Now, set $T=\fin{(A,\al)}{\al\in\om_1\ \wedge\ A\in\cA_\al}$, and let $\sqsubseteq$ be the binary relation on $T$ defined by
\[
(A',\al')\sqsubseteq(A'',\al'')\Longleftrightarrow(\al'\le\al''\ \wedge\ A'\supseteq A'').
\]
Observe that, since $\sqsubseteq$ is the restriction to $T$ of the product order of the two ordered sets $\big(\wp(X),\supseteq\big)$ and $(\om_1,\le)$, we have by general reasons that $\sqsubseteq$ is a (partial) order on $T$. We will further prove the following:
\vskip0.2cm\noi 
{\bf Fact.} $(T,\sqsubseteq)$ is a tree of height $\om_1$, and for every $\al\in\om_1$ we have that $\text{Lev}_\al(T)=\bfin{(A,\al)}{A\in\cA_\al}$.\par\noi
{\bf Proof.} Let $(\hat A,\hal)$ be an arbitrary element of $T$ (so that $\hat A\in
\cA_\hal$). By property (2), we have that for every $\al<\hal$ there exists an $A_\al\in\cA_\al$ with $\hat A\si A_\al$; since each $\cA_\al$ is a partition and $\hat A$, as an element of $\cA_\hal$, is nonempty, it is apparent that for every $\al<\hal$ the set 
$A_\al$ is the only element of $\cA_\al$ including $\hat A$. Then it easily follows that 
$\bfin{t\in T}{t\sqsubset(\hat A,\hal)}=\fin{(A_\al,\al)}{\al<\hal}$; moreover, by an argument similar to one of those used to prove property (1) when $\al$ is a limit ordinal 
$\lambda$, we may show that for every $\al',\al''$ with $\al'<\al''<\hal$ we have that 
$A_{\al''}\si A_{\al'}$ (actually, we know that $A_{\al''}$ must be included in some element $A'$ of $\cA_{\al'}$, and since $\cA_{\al'}$ is a partition and $A_{\al''}\cap A_{\al'}\supseteq\hat A\neq\vuo$, we conclude that $A'$ necessarily coincides with $A_{\al'}$). Therefore, the above-considered set $\bfin{t\in T}{t\sqsubset(\hat A,\hal)}=\fin{(A_\al,\al)}{\al<\hal}$, endowed with the order induced by $\sqsubseteq$, is similar to the ordinal $\hal$, and this proves at the same time that $(T,\sqsubseteq)$ is a tree and that $(\hat A,\hal)\in\text{Lev}_\hal(T)$. On the other hand, every element of $\text{Lev}_\hal(T)$ must have the second component equal to $\hal$---as for each $(A,\al)\in T$ with 
$\al\neq\hal$ we may prove as above that it belongs to the set $\text{Lev}_\al(T)$, which is disjoint from $\text{Lev}_\hal(T)$. Therefore $\text{Lev}_\hal(T)=\bfin{(A,\hal)}{A\in\cA_\hal}$, and this holds for each $\hal\in\om_1$. Of course, $\om_1$ turns out to be the height of $(T,\sqsubseteq)$.
\vskip0.2cm
Now we finish the proof of the proposition. First of all, observe that $\big\vert\cA_\al\big\vert\le\om_1$ for every $\al\in\om_1$; this is an elementary consequence of the fact 
that each $\cA_\al$ is an open partition of $X$ consisting of nonempty sets, and that $d(X)\le w(X)\le\om_1$. Since $\big\vert\text{Lev}_\al(T)\big\vert=\big\vert\bfin{(A,\al)}{A\in\cA_\al}\big\vert=\big\vert\cA_\al\big\vert\le\om_1$ for every $\al\in\om_1$, to show that 
$(T,\sqsubseteq)$ is a weak Kurepa tree we only have to prove that it has at least $\om_2$ branches.\par
For every $x\in X$ and $\al\in\om_1$, denote by $A_{x,\al}$ the only element of 
$\cA_\al$ containing $x$; notice that, since $x\in A_{x,\al'}\cap A_{x,\al''}\neq\vuo$ for every $\al',\al''\in\om_1$, we may prove as before that $A_{x,\al''}\si A_{x,\al'}$ for $\al'<\al''<\om_1$. Let, for every $x\in X$, $\Pi_x=\bfin{(A_{x,\al},\al)}{\al\in\om_1}$. Then 
$(A_{x,\al'},\al')\sqsubset(A_{x,\al''},\al'')$ for $\al'<\al''<\om_1$, so that $\Pi_x$ is a chain in $T$ intersecting all levels, i.e., $\Pi_x$ is a branch. Therefore, if we can prove that the association $x\mapsto \Pi_x$ is one-to-one, it will follow that in $T$ there are at least 
$\om_2$ branches. Indeed, suppose that $\bx,\by$ are arbitrary distinct elements in $X$; since $\cL$ is a base for $X$, which is T$_1$, we may consider an $\hal\in\om_1$ such that $\bx\in L_\hal$ and $\by\notin L_\hal$. Observe that, from $\bx\in A_{\bx,\hal+1}$, we have the inequality $L_{\hal}\cap A_{\bx,\hal+1}\neq\vuo$, and this implies by $(\ast)$ that $A_{\bx,\hal+1}\si L_\hal$; thus $\by\notin A_{\bx,\hal+1}$ (as $\by\notin L_\hal$), while of course $\by\in A_{\by,\hal+1}$. Then it follows that $A_{\bx,\hal+1}\neq A_{\by,\hal+1}$, and since $(A_{\by,\hal+1},\hal+1)$ is the only element of $\Pi_\by$ having the second component equal to $\hal+1$, we conclude that $(A_{\bx,\hal+1},\hal+1)\notin \Pi_\by$.
Therefore $\Pi_\bx\neq\Pi_\by$ (as the former set contains $A_{\bx,\hal+1}$ while the latter does not). 
\end{proof}

We have looked hard for \emph{small} non weakly Whyburn $P$-spaces in ZFC, but all we have are partial results. Here is another way to construct such an example, provided that the following innocent-looking question can be answered in the positive.

\begin{Quest} 
Does there exist in ZFC a dense-in-itself $P$-space of cardinality $\aleph_1$ every relatively discrete subset of which is closed?
\end{Quest}
 
Let $X$ be such a space and $Y=\omega_1\cup\{p\}$ be the one-point Lindel\"ofication of the discrete space $\omega_1$ and put $Z=X\times Y$. Fix an injective
mapping $f:X\to \omega_1$ and let $A=\{(x,f(x)): x \in
X\}\subseteq Z$. It is easy to realize that $A$ is not closed in
$Z$.  If for some $B \subseteq A$ we have $\overline B\setminus
A\neq\vuo$, then $(x,p)\in \overline B$  for some $x\in X$. Without any
loss of generality, we may assume $x \notin \pi_X(B)$. Since
the set $\pi_X(B)$ is not closed, it cannot be discrete so we may fix some $z\in\pi_X(B)$
which is not isolated in $\pi_X(B)$. It is not difficult
to see that $(z,p) \in \overline{B}$ so $Z$ is not weakly Whyburn.

Before going on to discuss Lindel\"of spaces we would like to recall another problem from \cite{PTTW} about which we know very little. Consider the space $\beta \omega$; under CH this space has character $\omega_1$ so its $\omega$-modification is even Whyburn, but what happens in a model where CH fails? 

\begin{Quest}
(\cite{PTTW}) Is the $\omega$-modification of $\beta \omega$ weakly Whyburn?
\end{Quest}

Problem 3.5 in \cite{PTTW} asks whether every (regular
Lindel\"of)
$P$-space is weakly Whyburn. Koszmider and Tall constructed
(\cite{KT}) a
model of ZFC+CH where there exists a regular Lindel\"of $P$-space of cardinality $\aleph_2$ without Lindel\"of subspaces of size $\omega_1$. Such a space cannot be weakly Whyburn, because of the following easy fact.

\begin{Prop}
Every Lindel\"of weakly Whyburn $P$-space $X$ of uncountable cardinality has a 
Lindel\"of subspace of cardinality $\omega_1$.
\end{Prop}
\begin{proof}
Pick a set $A \subseteq X$ with $\vert A\vert=\omega_1$. If $A$ is closed, then we are done.  Otherwise, there is some almost closed
$B\subseteq A$ converging outside it. Since in a $P$-space countable
sets are closed, the set $B$ must have cardinality $\omega_1$. Thus
$\overline{B}$ is a Lindel\"of subspace of $X$ of cardinality $\omega_1$.
\end{proof}

At least consistently, the assumption that $X$ is a $P$-space can
be dropped in the above proposition.

\begin{Prop}[{[{\sf CH}]}] 
Every Lindel\"of weakly Whyburn space $X$ of uncountable cardinality has a 
Lindel\"of subspace of cardinality $\omega_1$.
\end{Prop}
\begin{proof}
We will show that $X$ has a closed subset of size $\omega_1$.
Suppose that
it is not so and let $A \subseteq X$ be a set of size $\omega_1$. Then 
every almost closed set converging outside $A$ is countable, so that the
Whyburn-closure of $A$ has cardinality $\omega_1^\omega=\omega_1$. Let
$A^\alpha$ be the $\alpha$-iterate of the Whyburn closure of $A$.
Then
$B=\bigcup_{\alpha \in \omega_1} A^\alpha$ has cardinality
$\omega_1$; if
this last set were non-closed then we could find some almost
closed $P$
converging outside it. Thus, $P$ would be countable and, by the regularity of 
$\omega_1$, there would be some $\beta \in \omega_1$ such that $P \subseteq \bigcup_{\alpha\in\beta}A^\alpha$. Hence the limit of $P$ would be inside $B$, which is
a contradiction.
\end{proof}

As a corollary we get a proposition estabilished in \cite{BT}
with the aid of elementary submodels:

\begin{Th}[{[{\sf CH}]}] (\cite{BT}) Every Hausdorff Lindel\"of sequential space of uncountable cardinality has a Lindel\"of subspace of cardinality $\omega_1$.
\end{Th}

\begin{Quest}
Is it true in ZFC that every Lindel\"of weakly Whyburn space of uncountable cardinality has a Lindel\"of subspace of cardinality $\omega_1$?
\end{Quest}

This is known to be true in the special case of regular Lindel\"of scattered spaces by a result of Ofelia T. Alas (see \cite[Theorem~10]{BT}).

It is known (see \cite{B}) that compact weakly Whyburn spaces are pseudoradial. We don't know whether the same holds for Lindel\"of weakly Whyburn $P$-spaces, but we  prove the following partial result.

\begin{Th}
Every regular Lindel\"of weakly Whyburn $P$-space of
pseudocharacter less
than $\aleph_\omega$ is pseudoradial.
\end{Th}
\begin{proof}
Let $A \subseteq X$ be a non closed set and $B \subseteq A$ such that
$\overline{B} \setminus A=\{x\}$ for some $x \in X$. Every
regular
$P$-space is zero-dimensional, so let $\mathcal{U}=\{U_\alpha:
\alpha <
\kappa \}$ be a family of clopen sets in $\overline{B}$, 
such that
$\bigcap_{\alpha \in \kappa} U_\alpha=\{x\}$, and having the minimal cardinality among all families with the previous properties. Since we are in a
$P$-space
and $x$ is an accumulation point of $B$, the cardinal $\kappa$ must be
uncountable. Use minimality of $\mathcal{U}$ to pick $x_\gamma
\in
\bigcap_{\alpha <\gamma} U_\alpha \setminus \{x\}$ for each
$\gamma \in
\kappa$. Since $\overline{B}$ is Lindel\"of and $\kappa$ is a
regular
uncountable cardinal, the sequence $\{x_\gamma : \gamma<\kappa\}
\subseteq A$
converges to $x$.
\end{proof}

\begin{Quest}
Does there exist a Lindel\"of weakly Whyburn non-pseudoradial $P$-space?
\end{Quest}

A negative answer to the above question would exhibit an interesting pathology in the class of Lindel\"of $P$-spaces. We now suggest a possible way to look for a counterexample. Kunen has constructed a compact space $X$ and a point $p \in X$ with the following properties:

\begin{enumerate}
\item{$\chi(p, X)>\omega$.}
\item{No sequence of uncountable regular length converges to $p$ in $X$.}
\end{enumerate}

Suppose that a compact \emph{scattered} space $Y$ having the above properties exists, then it is easy to construct a weakly Whyburn non pseudoradial Lindelof $P$-space. Indeed, since $Y$ is compact scattered its omega-modification $Y_\delta$ is a Lindel\"of $P$-space (see \cite{A}, Lemma II.7.14). Since $Y_\delta$ is regular and scattered it is weakly Whyburn. By 2) no sequence will converge to $p$ in $Y_\delta$. (Since character equals pseudocharacter in compact spaces 1) will guarantee that $p$ is not isolated in $Y_\delta$). 

\section{Product of weakly Whyburn $P$-spaces}

The behaviour of the weakly Whyburn property under the product operation is not very clear. The first listed author has proved that the product of a compact semiradial space and a compact weakly Whyburn space is weakly Whyburn but we are not aware of a non weakly Whyburn product of compact weakly Whyburn spaces. We don't know whether the same holds replacing compact space with Lindel\"of $P$-space but at least we prove that the product of Lindel\"of weakly Whyburn $P$-space and a Lindel\"of Whyburn $P$-space is weakly Whyburn (see Corollary~\ref{uno}).

Recall that a space has countable extent if each of its closed discrete subsets is countable. Every Lindel\"of space has countable extent. So the next lemma is a strengthening of Theorem 2.5 in \cite{PTTW}.   

\begin{Lemma} \label{lem: extent}
If $X$ is a Whyburn $P$-space of countable extent, then $X$ is a
radial space of radial character $\aleph_1$.
\end{Lemma}
\begin{proof}Assume that  $A\subseteq X$  and  $x
\in \overline{A} \setminus A$. Fix an almost closed $P\subseteq 
A$ such  that  $x \in \overline{P}$. Take a maximal disjoint
family $\gamma$ of open subsets of $P$ whose closures do not
contain $x$.  Then $x \in \overline {\bigcup \gamma}$ and hence
there is an  almost closed $Q \subseteq \bigcup \gamma$ such 
that  $x \in \overline{Q}$. Let $\gamma'=\{U \in \gamma: U \cap Q
\neq \emptyset \}$. Since $X$ is a $P$-space, it is easy  to  see
that $\gamma'$ is  uncountable. For each $U \in \gamma'$ take an
$x_U \in U \cap Q$. The set $B=\{x_U: U \in \gamma'\}$  is 
discrete and $(\overline{B} \setminus B) \cap  (\bigcup
\gamma)=\emptyset$. This means $B$ is closed in $Q$ and therefore
$x$ has to be the only accumulation point of $B$ in the space $Q
\cup \{x\}$. $B \cup\{x\}$ is an uncountable space with the unique
non-isolated point $x$. Being of countable extent, such a space
must be Lindel\"of and of cardinality $\aleph_1$. Hence $B$ is a sequence of length 
$\aleph_1$ which converges to $x$. 
\end{proof}

\begin{Prop}
A $P$-space $X$ of cardinality $\aleph_1$ and countable extent is
Lindel\"of and radial.
\end{Prop}
\begin{proof}
We first prove that $X$ is Lindel\"of. Toward a contradiction, suppose there exists an open cover $\cA$ of $X$ with no countable subcover. Of course, since $\vert X\vert=\om_1$, we may assume that $\vert\cA\vert=\om_1$. Also, since $\cA$ has no countable subcover, we can further assume (up to passing to a suitable subcover) that $\cA$ is indexed as  
$\fin{A_\al}{\al\in\om_1}$, in such a way that  
\[
\forall\al\in\om_1\colon A_\al\stm\bigcup_{\al'<\al}A_{\al'}\neq\vuo.
\]
For every $\al\in\om_1$, pick an $x_\al\in A_\al\stm\bigcup_{\al'<\al}A_{\al'}$; we claim that the set $D=\fin{x_\al}{\al\in\om_1}$ is closed and discrete, and this will contradict the fact that $X$ has countable extent---as clearly $\al\mapsto x_\al$ is one-to-one.\par
To prove the discreteness of $D$, note that for every $\al\in\om_1$ the set $A_\al$ is a neighbourhood of $x_\al$ missing all points $x_{\al'}$ with $\al'>\al$; also, since the set
$C_\al=\fin{x_{\al'}}{\al'<\al}$ is closed (because $X$ is a $P$-space), we see that $A_\al\stm C_\al$ is an open neighbourhood of $x_\al$ whose intersection with $D$ gives $\{x_\al\}$. Suppose now $\hx\in X\stm D$; then $\hx\in A_\hal$ for some $\hal\in\om_1$. Arguing essentially as before, we see that $A_\hal\stm\fin{x_{\al'}}{\al'\le\hal}$ is an open neighbourhood of $\hx$ missing $D$; thus $D$ is closed.\par
Notice that $X$ is Hausdorff and, being a Lindel\"of $P$-space, it is also regular. 
To prove that $X$ is radial, due to Lemma~\ref{lem: extent} it will suffice to show that 
$X$ has Whyburn property. Let $a\in X$ and $M\si X$ be such that $a\in\overline M\stm M$; index $X\stm\{a\}$ as $\fin{x_\al}{\al\in\om_1}$ and consider a decreasing 
$\om_1$-sequence $\{V_\al\}_{\al\in\om_1}$ of open neighbourhoods of $a$ such that 
\begin{equation}\label{val}
\fore\al{\om_1}\ \overline{V_\al}\cap\fin{x_{\al'}}{\al'\le\al}=\vuo;
\end{equation}
let also $y_\al$ be an element of $V_\al\cap M$ for every $\al\in\om_1$, and set $L=\fin{y_\al}{\al\in\om_1}$. We claim that $L$ is an almost-closed set converging to $a$. Indeed, on the one hand if $z$ is an arbitrary element of $X\stm(L\cup\{a\})$ then $z=x_\hal$ for some $\hal\in\om_1$, so that $z\in X\stm\overline{V_\hal}$; since $z\notin L$, it also follows that $U=X\stm(\fin{y_\al}{\al<\hal}\cup\overline{V_\hal})$ is an open neighbourhood of $z$, and $U$ is disjoint from $L$ as $L=\fin{y_\al}{\al\in\om_1}\si\fin{y_\al}{\al<\hal}\cup\overline{V_\hal}$. On the other hand, to show that $a\in\overline L$, consider an arbitrary neighbourhood $W$ of $a$; clearly, if we can prove that $V_\al\si
W$ for some $\al\in\om_1$, then the inequality $W\cap L\neq\vuo$ will follow. 
Actually, consider the closed subset $T=X\stm W$ of $X$, and let $\cA=\fin{X\stm\overline{V_\al}}{\al\in\om_1}$. We see that $\cA$ covers $T$, as every element of $T$ must be $x_\al$ for some $\al\in\om_1$, and $x_\al\in X\stm\overline{V_\al}$ by \eqref{val}. Since $T$ is Lindel\"of, there must exist a countable subset $A$ of $\om_1$ such that $X\stm W\si\bigcup_{\al\in A}(X\stm\overline{V_\al})=X\stm\bigcap_{\al\in A}\overline{V_\al}$; using the decreasing characters of the neighbourhoods $V_\al$, it also follows that $X\stm W\si X\stm\overline{V_\hal}$---where $\hal=\sup A$---whence $V_\hal\si\overline{V_\hal}\si W$. 
\end{proof}

\begin{Th} \label{th: prod}
The product of a  pseudoradial $P$-space $X$ of radial character $\omega_1$ with a weakly Whyburn $P$-space $Y$ of countable extent is weakly Whyburn.
\end{Th}
\begin{proof}
Let $A \subseteq X\times Y$ be a non-closed subset. Take
$(x,y) \in \overline {A} \setminus A$.
If $A \cap(\{x\}\times Y)$ were non-closed, we could
immediately apply the  weak Whyburn property of $Y$ to find a set
$B\subseteq A$ satisfying $\overline{B}\setminus A=\{(x,y)\}$.
So,  we may assume that $A \cap(\{x\}\times Y)$ is closed.
By passing to a suitable subset, we may further assume that
$A \cap(\{x\}\times Y)= \emptyset$.
Since $x \in \overline {\pi_X (A)}$,
it follows that $\pi_X (A)$ is not closed in $X$.
So, there is a convergent sequence $\{x_\alpha : \alpha  \in
\omega_1\} \subseteq
\pi_X (A)$
with limit outside $\pi_X (A)$.
For each $\alpha  \in \omega_1$ take a point $y_\alpha  \in Y$ in
such a way
that
$(x_\alpha , y_\alpha ) \in A$.

\noindent Case 1: The set $\{y_\alpha  :\alpha \in \omega_1\}$
is countable.  There exists an uncountable subset $S\subseteq
\omega_1$ such that $y_\alpha=y_\beta$ for any $\alpha ,\beta \in
S$. Then,  we put $B=\{(x_\alpha ,y_\alpha ) :\alpha \in  S\}$. 

\noindent Case 2: The set $\{y_\alpha :\alpha \in \omega_1\}$ is
uncountable.    
Since $Y$ has countable extent,  we may assume---up to removing one point---that the set $\{y_\alpha :\alpha \in
\omega_1\}$ is not closed. Therefore, there exists an uncountable set
$S\subseteq \omega_1$ such that the set $C=\{y_\alpha :\alpha \in
S\}$ is almost closed. As the set $cl_Y (C)$ has still countable
extent, by Lemma $\ref{lem: extent}$ there exists an (uncountable) set $T\subseteq S$ such that
the set $\{y_\alpha :\alpha \in T\}$ is a convergent sequence.
Then put $B=\{(x_\alpha, y_\alpha ) :\alpha \in T\}$.

In both of the previous cases, the set $B$ is a convergent
sequence   which is also an almost closed subset of $A$.
\end{proof}

\begin{Cor}\label{uno} 
The product of a Whyburn Lindel\"of $P$-space with a weakly Whyburn Lindel\"of $P$-space is weakly Whyburn.
\end{Cor}

\begin{Quest}
Is the product of two Lindel\"of weakly Whyburn $P$-spaces always
weakly Whyburn?
\end{Quest}

%It is worth observing, as a result of some independent interest, that for a space of cardinality $\aleph_1$ the conclusion of  Lemma~\ref{lem: extent} holds even without assuming the Whyburn property.

Constructing a non-Whyburn product of two Lindel\"of $P$-spaces
turned out not to be so easy. Since every compact Whyburn space
is Fr\'echet, the analogous problem for compact spaces is the
construction of two compact Fr\'echet spaces whose product is not
Fr\'echet. This is usually tackled using the one-point compactification of Mrowka spaces on suitable almost disjoint families on $\omega$ (see \cite{BR} and \cite{S2}), so it was natural to try and extend that approach.

We first recall that a family $\A$ of uncountable subsets of
$\omega_1$ is said to be \emph{almost disjoint} if $|a \cap
b|<\omega_1$ for all distinct $a,b \in \A$; it is said to be a MAD family
if it is maximal with respect to this property.
Define now a topology on $X=\omega_1 \cup \A$ in the following
way. Each point of $\omega_1$ is isolated and a neighbourhood of
$p \in \A$ is $\{p\} \cup p$ minus a countable set. The space $X$
is locally Lindel\"of so we may take its one-point
Lindel\"ofication $\L(\A)$; in concrete terms, $\L(\A)$ is obtained by adding a point $\infty$ to $\cA$, and is endowed with a topology whose restriction to $X$ coincides with the original topology, while the basic neighbourhoods of $\infty$ are all sets of the form $\{\infty\}\cup(\cA\stm\cF)\cup\big(\bigcup(\cA\stm\cF)\stm F\big)$, where $\cF\in[\cA]^{<\omega_1}$ and $F\in[\omega_1]^{<\omega_1}$. Notice that, this way, $\cL(
\cA)$ turns out to be a Lindel\"of (regular) P-space.\par 
Let $\A$ be an almost disjoint family on $\omega_1$ and let
$\mathcal{I}(\A)$ be the ideal of those uncountable subsets of
$\omega_1$ which can be {\it almost-covered} by a countable subfamily of $\A$ (which means that the set-theoretic difference between the former and the union of the latter one is countable); let also $\cI^+(\cA)=[\om_1]^{\om_1}\stm\cI(\cA)$. Furthermore, let $\mathcal{M}(\A)$ be the set of those $X \in
[\omega_1]^{\omega_1}$ such that $\A_X=\{X \cap p : p \in \A,\,\text{$X\cap p$ is uncountable} \}$ is a MAD family on $X$. We introduce the following:

\begin{Def}
An almost disjoint family $\A$ on $\omega_1$ is said to be
\emph{nowhere MAD} if $\M(\A) \subseteq \I(\A)$.
\end{Def}

Now we will link some topological features of $\L(\A)$ to the
combinatorial structure of $\A$ via the previous definition.

\begin{Th}
Let $\A$ be an almost disjoint family on $\omega_1$. The space
$\L(\A)$ is Whyburn if and only if $\A$ is nowhere MAD.
\end{Th}
\begin{proof}
Suppose first that there is some $X \in \M(\A) \cap{\mathcal I}^+(\A)$,
then $\infty \in \overline{X}$ (where $\infty$ is the point added to $\cA$, to obtain $\cL(\cA)$). We will show that $X$ contains no
almost closed set converging to $\infty$. Indeed let $Y \subseteq X$ be 
such that $\infty \in \overline{Y}$; then $Y$ is clearly uncountable, hence---since the family $\A_X$ is MAD on $X$---there must be some $p \in \A$ such that $p \cap Y$ is
uncountable. So $p \in \overline{Y}$ and we are done. To prove
the converse note that at each $q \in \A$ the Whyburn property is clearly satisfied, therefore we only need to check it at $\infty$. Suppose that $\A$ is
nowhere MAD and pick some $X\subseteq\om_1\cup\cA$ such that $\infty
\in \overline{X}$. If $\infty\in\overline{X\cap\cA}$ then clearly $\overline{X\cap\cA}=
(X\cap\cA)\cup\{\infty\}$, so we may restrict ourselves to the case where $X\si\om_1$. 
Now, $\infty\in\overline{X}$ implies that $X\notin\cI(\cA)$ (in particular, of course, $X$ is uncountable), so by nowhere MADness of $\A$ we can find some uncountable $Y \subseteq X$ such that $Y \cap p$ is countable for each $p \in \A$. Thus
$Y$ is an almost closed set such that $Y \to \infty$.
\end{proof}

\begin{Lemma}\label{lem}
Let $\cA$ be a MAD family on a set $X$ of cardinality $\omega_1$, and $\{\cA_1,\cA_2\}$ 
be a partition of $\cA$, with $\cA_1$ nowhere MAD on $X$ and $\vert\cA_1\vert\ge\omega_1$. Suppose also to have associated to every $A\in\cA_1$ an element $M(A)$ of $[A]^{\omega_1}$. Then there exists $\tilde A\in\cA_2$ such that $\tilde A\cap\big(\bigcup_{A\in\cA'}M(A)\big)\neq\vuo$.  
\end{Lemma}
\begin{proof}
Let, for the sake of simplicity, $M^\ast=\bigcup_{A\in\cA_1}M(A)$, and notice that $M^\ast\in\cJ^+(\cA_1)$. Indeed, if $\cF$ is a countable subfamily of $\cA_1$, then there exists an $\hat A\in\cA_1\stm\cF$. Since $\cA_1\si\cA$ and $\cA$ is almost disjoint, it follows that $\big\vert\hat A\cap\big(\bigcup\cF\big)\big\vert<\omega_1$. Therefore, in particular, $M\big(\hat A\big)\stm\bigcup\cF\neq\vuo$, whence also $M^\ast\stm\bigcup\cF\neq\vuo$.\par
Now, since $\cA_1$ is nowhere MAD, there must exist $S\in[M^\ast]^{\omega_1}$ such that the set $S\cap(A\cap M^\ast)=S\cap A$ is countable for every $A\in\cA_1$. On the other hand, since $S\in[M^\ast]^{\omega_1}\si[X]^{\omega_1}$ and $\cA$ is MAD on 
$X$, there will be $\tA\in\cA$ with $\big\vert\tA\cap S\big\vert=\omega_1$; therefore, such $\tA$ must belong to $\cA_2$, and of course we have in particular $\tA\cap M^\ast\supseteq\tA\cap S\neq\vuo$.
\end{proof}

\begin{Cor}\label{cor}
Let, as in the previous lemma, $\cA$ be a MAD family on a set $X$ of cardinality 
$\omega_1$, and $\{\cA_1,\cA_2\}$ be a partition of $\cA$, with $\cA_1$ nowhere MAD on $X$ and of cardinality $\omega_1$. Then there exists $A\in\cA_1$ such that $\big\vert A\cap\big(\bigcup\cA_2\big)\big\vert=\omega_1$. 
\end{Cor}
\begin{proof}
Toward a contradiction, assume $\big\vert A\cap\big(\bigcup\cA_2\big)\big\vert<\omega_1$ for every $A\in\cA_1$; letting, for every $A\in\cA_1$, $M(A)=A\stm\big(\bigcup\cA_2\big)$, we see that each $M(A)$ is uncountable, hence by 
Lemma~\ref{lem} it follows that $\tA\cap\big(\bigcup_{A\in\cA_1}M(A)\big)\neq\vuo$
for some $\tA\in\cA_2$. Clearly, this is a contradiction, as each $M(A)$ is disjoint from 
$\bigcup\cA_2$. 
\end{proof}

\begin{Th}
Let $\A$ be an almost disjoint family on $\omega_1$, and $\{\cA_1,\cA_2\}$ be a partition of $\cA$ into two uncountable subfamilies. If $\A$ is not nowhere MAD while $\cA_1$ is nowhere MAD, then $\L(\A_1) \times\L(\A_2)$ is not a Whyburn space.
\end{Th}
\begin{proof}
Let $X \in \M(\A) \cap \I^+(\A)$. Consider the diagonal
$\Delta=\{(\alpha, \alpha) : \alpha \in X \}$; then $(\infty,
\infty) \in \overline{\Delta}$. To verify this fact, consider a basic neighbourhood $W$ of 
$(\infty,\infty)$ in $\L(\A_1)\times\L(\A_2)$ of the form 
\[
W=\Big(\{\infty\}\cup(\cA_1\stm\cF_1)\cup\big(\textstyle{\bigcup}(\cA_1\stm\cF_1)\stm F_1\big)\Big)\times\Big(\{\infty\}\cup(\cA_2\stm\cF_2)\cup\big(\textstyle{\bigcup}(\cA_2\stm\cF_2)\stm F_2\big)\Big),
\]
where $\cF_i\in[\cA_i]^{<\om_1}$ and $F_i\in[\om_1]^{<\om_1}$ for $i=1,2$; since $X\notin\cI(\cA)$, the set $X\stm\big(\big(\bigcup\cF_1\big)\cup\big(\bigcup\cF_2\big)\cup F_1\cup F_2\big)$ must be nonempty, and taking any $\al$ in it we easily see that $(\al,\al)\in\Delta\cap W$.\par
Now, let $B=\{(\alpha, \alpha) :
\alpha \in Y \} \subseteq \Delta$ be such that $(\infty, \infty)
\in \overline{B}$, with $Y \subseteq X$. Then $Y$ must be uncountable, and the family
$\cA_Y=\fin{A\cap Y}{A\in\cA,\,\vert A\cap Y\vert=\omega_1}$ is MAD on $Y$ (as the restriction of $\cA$ to $X$ is MAD on $X$), while the family $\cA_1'=\big\{A\cap Y\,\colon\,A\in\cA_1,\,\vert A\cap Y\vert=\omega_1\big\}$ is nowhere MAD on $Y$ (as $\cA_1$ is nowhere MAD on $\omega_1$). Letting also $\cA_2'=\fin{A\cap Y}{A\in\cA_2,\,\vert A\cap Y\vert=\omega_1}$, we see that $\{\cA_1',\cA_2'\}$ is a partition of $\cA_Y$, so that by Corollary~\ref{cor} there exists $A^\ast\in\cA_1'$ such that $\big\vert A^\ast\cap\big(\bigcup\cA_2'\big)\big\vert=\omega_1$. Let $p\in\cA$ be such that $A^\ast=p\cap Y$; we claim that 
$(p,\infty)\in\overline B$.\par
Indeed, consider a basic neighbourhood of $(p,\infty)$ of the form 
\[
\big(\{p\}\cup(p\stm F_1)\big)\times\Big(\{\infty\}\cup(\cA_2\stm\cF_2)\cup\big(\textstyle{\bigcup}(\cA_2\stm\cF_2)\stm F_2\big)\Big),
\]
with $\cF_2\in[\cA_2]^{<\omega_1}$ and $F_1,F_2\in[\omega_1]^{<\omega_1}$. Then from $\big\vert A^\ast\cap\big(\bigcup\cA_2'\big)\big\vert=\omega_1$ it follows that $\big\vert p\cap\big(\bigcup\cA_2\big)\cap Y\big\vert=\omega_1$; moreover, since $\vert 
A\cap p\vert<\omega_1$ for every $A\in\cF_2$, it also follows that $\big\vert p\cap\big(\bigcup(\cA_2\stm\cF_2)\big)\cap Y\big\vert=\omega_1$, and finally that $\big\vert (p\stm F_1)\cap\big(\bigcup(\cA_2\stm\cF_2)\stm F_2\big)\cap Y\big\vert=\omega_1$. Then taking $\hal\in(p\stm F_1)\cap\big(\bigcup(\cA_2\stm\cF_2)\stm F_2\big)\cap Y$, we see that 
\[
(\hal,\hal)\in\Big(\big(\{p\}\cup(p\stm F_1)\big)\times\Big(\{\infty\}\cup(\cA_2\stm\cF_2)\cup\big(\textstyle{\bigcup}(\cA_2\stm\cF_2)\stm F_2\big)\Big)\Big)\cap B.
\]
\end{proof}

Thus to find a non-Whyburn product of Lindel\"of Whyburn
$P$-spaces it suffices to construct a MAD family on $\omega_1$
which can be split into two nowhere MAD families. This can be
done under a proper set-theoretic assumption, by a well-known line of reasoning (see \cite{BR}).

\begin{Th}[{[$2^{\aleph_1}=\aleph_2$]}] There is a MAD family on $\omega_1$
which can be split into two nowhere MAD subfamilies.
\end{Th}
\begin{proof}
Let $\A=\{A_\alpha : \alpha<\omega_2\}$ be an enumeration of $[\om_1]^{\om_1}$. 
We will construct almost disjoint families
$\{\S_\alpha: \alpha < \omega_2\}$ of uncountable subsets of
$\omega_1$ with the property that for each $\beta \leq \alpha$
either $A_\beta$ is covered by a countable subcollection of
$\S_\alpha$ or there is some $S \in\S_\alpha$ such that $S
\subseteq A_\beta$. Let $\S_0'$ be any partition of $\omega_1$
into $\omega_1$ many uncountable sets. If $A_0$ is covered by a
countable subfamily of $\S_0'$ , let $\S_0=\S_0'$, else let
$\S_0'=\{T_\alpha: \alpha \in \omega_1\}$ be an enumeration; then
for each $\beta \in \omega_1$ the set $A_0 \setminus
\bigcup_{\alpha \leq \beta} T_\alpha$ is uncountable. Pick $a_0
\in A_0 \setminus T_0$ and for each $\beta \neq 0$ pick
$$a_\beta \in A_0 \setminus\Big(\bigcup_{\alpha \leq \beta} T_\alpha
\cup \{a_\alpha : \alpha < \beta\}\Big).$$
Let $S_0=\{a_\alpha : \alpha \in \omega_1\}$ and set $\S_0=\S_0'
\cup \{S_0\}$. Clearly, $S_0\si A_0$ and $\cS_0$ is almost disjoint.\par
If $\S_\beta$ is chosen for each $\beta<\alpha$,
$\S_\alpha'=\bigcup_{\beta<\alpha} \S_\beta$. If $A_\alpha$ is
not covered by a countable subfamily of $\S_\alpha'$ let
$\S_\alpha'=\{T^{\alpha}_\gamma : \gamma \in \omega_1\}$ be an
enumeration; then, for each $\beta \in \omega_1$ the set
$A_\alpha \setminus \bigcup_{\gamma \leq \beta}
T^{\alpha}_\gamma$ is uncountable. Pick $a_0^\alpha \in A_\alpha
\setminus T_0^\alpha$ and for each $\beta \neq 0$ pick

$$a_\beta^\alpha \in A_\alpha \setminus\Big(\bigcup_{\gamma \leq
\beta} T_\beta^\alpha \cup \{a_\gamma^\alpha : \gamma<\beta\}\Big).$$
Let $S_\alpha=\{a_\gamma^\alpha : \gamma \in \omega_1\}$ and let
$\S_\al=\S_\alpha' \cup \{S_\alpha\}$. It is clear that each $\S_\alpha$
has the property stated at the beginning and is almost disjoint;
note also that $\S_\beta \subseteq\S_\alpha$ whenever $\beta<\alpha$. Put 
$\S'=\bigcup_{\alpha \in \omega_2} \S_\alpha$; since
$\S'$ is almost disjoint there is some MAD family $\S$ such that
$\S' \subseteq \S$. The family $\S$ inherits from $\S'$ the property that if $X
\subseteq \omega_1$ then either $X$ is covered by a countable
subfamily of $\S$ or there is some $S \in\cS$ such that $S
\subseteq X$. For each $S \in \S$ let $\{S^+,S^\ast\}$ be a partition of $S$ into two sets of cardinality $\om_1$; then set $\S^\ast=\{S^\ast : S \in \S\}$ and let $\S^+=\{S^+ : S \in
\S\}$. We see that $\T=\S^\ast \cup \S^+$ is a MAD family on $\omega_1$
with the property that if $X \subseteq \omega_1$ is not covered
by a countable subfamily of $\T$ then there are disjoint $S_1 \in
\S^\ast$ and $S_2 \in \S^+$ with $S_1 \cup S_2 \subseteq X$.
This clearly implies that $\S^\ast$ and $\S^+$ are nowhere MAD.
\end{proof}

\begin{Cor}[{[$2^{\aleph_1}=\aleph_2$]}]
There exist two Lindel\"of Whyburn P-spaces whose product is not Whyburn.
\end{Cor}

\begin{Quest}
Does there exist in ZFC a non-Whyburn product of Lindel\"of
Whyburn $P$-spaces?
\end{Quest}

\section{Acknowledgements}
The authors are indebted to Eva Murtinov\'a for several helpful comments, that improved the exposition of the paper. The third author is indebted to Frank Tall for pointing him to Shelah and Todorcevic's paper and to the members of the Auburn General Topology seminar, especially to Stu Baldwin and Gary Gruenhage, for some helpful comments during a talk where a preliminary version of this paper was presented. Finally, the authors are very grateful to the referee for his really careful reading of the paper, and for a number of suggestions which have been inserted in the present version.

\end{document}